\documentclass{amsart}

\usepackage{amsmath,amsfonts,amssymb,color,hyperref, enumerate}

\newtheorem{theorem}{Theorem}[section]
\newtheorem{lemma}[theorem]{Lemma}
\newtheorem{proposition}[theorem]{Proposition}

\newtheorem{definition}[theorem]{Definition}
\newtheorem{corollary}[theorem]{Corollary}

\newtheorem{conjecture}[theorem]{Conjecture}
\newtheorem{question}[theorem]{Question}
\newtheorem{problem}[theorem]{Problem}

\newcommand{\Sym}{\mathop{\mathrm{Sym}}}
\newcommand{\Alt}{\mathop{\mathrm{Alt}}}
\newcommand{\PSL}{\mathop{\mathrm{PSL}}}

\newcommand{\Aut}{\mathop{\mathrm{Aut}}}

\newcommand{\Cay}{\mathop{\mathrm{Cay}}}

\def\nor#1#2{{\bf N}_{{#1}}({{#2}})}

\begin{document}

\title[Partitions derangement graph]{On triangles in  derangement graphs}  

\author[K.~Meagher]{Karen Meagher}
\address{Department of Mathematics and Statistics, University of Regina, Regina, Saskatchewan
S4S 0A2, Canada}
\email{meagherk@uregina.ca}

\author[A.~Sarobidy Razafimahatratra]{Andriaherimanana Sarobidy Razafimahatratra}
\address{Department of Mathematics and Statistics, University of Regina, Regina, Saskatchewan
S4S 0A2, Canada}
\email{sarobidy@phystech.edu}

\author[P.~Spiga]{Pablo Spiga}
\address{Dipartimento di Matematica e Applicazioni, University of Milano-Bicocca, Via Cozzi 55, 20125 Milano, Italy} 
\email{pablo.spiga@unimib.it}

\begin{abstract}
  Given a permutation group $G$, the derangement graph $\Gamma_G$ of $G$ is the Cayley graph with connection set the set of all derangements of $G$.  We prove that,  when $G$ is transitive of degree at least $3$, $\Gamma_G$ contains a triangle. 

The motivation for this work is the question of how large can be the ratio of the independence number of $\Gamma_G$ to the size of the stabilizer of a point in $G$. We give examples of transitive groups   where this ratio is maximum. 
  
\end{abstract}


\maketitle

\section{Introduction}\label{sec:intro}

The Erd\H{o}s-Ko-Rado theorem is a fundamental result of extremal combinatorics.

\begin{theorem}[\cite{erdos1961intersection}]
Let $\Omega$ be a set of cardinality $n$, let $k$ be a positive integer with $2k\le n$ and let $\mathcal{F}$ be a family of $k$-subsets of $\Omega$ with $A\cap B\ne\emptyset$, for all $A,B\in\mathcal{F}$. Then $|\mathcal{F}|\leq \binom{n-1}{k-1}$. Moreover, provided $2k<n$, equality holds if and only if all elements in $\mathcal{F}$ contain a fixed element of $\Omega$.\label{theorem:EKR-set}
\end{theorem}

Theorem~\ref{theorem:EKR-set} can be extended for various objects, including permutation groups, see for instance~\cite{6,deza1978intersection,godsil2016algebraic,10}. Given a permutation group $G$ on a set $\Omega$ and $\mathcal{F} \subseteq G$, we say that $\mathcal{F}$ is \textbf{\textit{intersecting}} if any two permutations $g,h \in \mathcal{F}$ agree on some $\omega\in \Omega$, that is, $\omega^g=\omega^h$. Given $\omega,\omega'\in \Omega$, the set $G_{\omega \rightarrow \omega'}:=\{g\in G\mid \omega^g=\omega'\}$ of all permutations in $G$ that map $\omega$ to $\omega'$ is intersecting;  we call the intersecting sets of this type the \textbf{\textit{canonical intersecting sets}}. Clearly, $G_{\omega \rightarrow \omega'}$ is either empty or a  right coset of the stabilizer $G_\omega$ of the point $\omega$ in $G$. 

A transitive group $G$ has the \textbf{\textit{Erd\H{o}s-Ko-Rado property}} or \textbf{\textit{EKR-property}} if the maximum cardinality of an intersecting family is  $\frac{|G|}{|\Omega|}$. Moreover, if equality only holds for canonical intersecting sets, then we say that $G$ has the \textbf{\textit{strict EKR-property}}.
For instance, it was proved independently by Cameron and Ku~\cite{6} and by Larose and Malvenuto~\cite{10} that $G=\Sym(\Omega)$ has the strict EKR-property. However, there are many  interesting permutation groups that have the EKR-property but not the strict EKR-property, see for example~\cite{MeSp,spiga}.

Given a permutation group $G$ on $\Omega$, we let $\mathcal{D}$ be the set of all derangements of $G$, where a \textbf{\textit{derangement}} is a  permutation without fixed points. The \textbf{\textit{derangement graph}} of $G$ is the graph $\Gamma_G $ whose vertex set is the set $G$ and whose edge set consists of all pairs $(h,g)\in G\times G$ such that $gh^{-1}\in \mathcal{D}$. In particular, $\Gamma_G$ is the Cayley graph of $G$ with connection set $\mathcal{D}$. Note that $\Gamma_G$ is  loop-less because $\mathcal{D}$ does not contain the identity element of $G$ and is a simple graph because $\mathcal{D}$ is inverse-closed, that is, $\mathcal{D}=\{g^{-1}\mid g\in \mathcal{D}\}$. With this terminology, an intersecting family of $G$ is an \textbf{\textit{independent set}} or \textbf{\textit{coclique}} of $\Gamma_G$, and vice versa. 

Let $\omega\in \Omega$ with $G_{\omega}$ having maximum cardinality among point stabilizers. The \textbf{\textit{intersection density}} or \textbf{\textit {EKR-density}} of the intersecting family $\mathcal{F}$ of $G$  is defined by
\[
\rho(\mathcal{F}):= \frac{|\mathcal{F}|}{| G_{\omega}|}.
\]
The \textbf{\textit{intersection density}} of $G$ is 
\[
\rho(G):= \max\left\{\rho(\mathcal{F})\ \mid \mbox{$\mathcal{F}\subseteq G$,\,$\mathcal{F}$ is intersecting} \right\}.
\]
This parameter was defined by Li, Song and Pantangi in~\cite{li2020ekr} (actually, in \cite{li2020ekr}, our $\rho(G)$ was denoted by $\rho(G,\Omega)$). In this paper we consider transitive groups, so $|G_{\omega}| = |G|/|\Omega|$.
Observe that a transitive group $G$ has the EKR-property if and only if $\rho(G) = 1$.  As Li, Song and Pantangi point out, this ratio is a measure of how far a group is from satisfying the EKR-property. For any group $G$, $\rho(G) \geq 1$. Further, if $G$ is transitive and $|\Omega|\ge 2$, then by Jordan's theorem $G$ has a derangement, which implies $\rho(G) < |\Omega|$.

Our main motivation in this paper is to find groups that are very far from having the EKR-property, that is, groups with large intersection density.  Li, Song and Pantangi conjectured in \cite{li2020ekr} that the intersection density is at most $\sqrt{|\Omega|}$ and they constructed a transitive group $G$ with $\rho(G)\sim \sqrt{|\Omega|}$. In Theorem~\ref{thrm:enu}, we give examples of transitive groups with intersection density larger than what Li, Song and Pantangi have conjectured in~\cite{li2020ekr}.


\begin{definition}{\rm 
	Let $n\geq 2$. Define 
	  $$\mathcal{I}_n := \left\{ \rho(G) \mid \mbox{$G$ transitive of degree $n$}\right\}.$$ 
	The set $\mathcal{I}_n$ is a finite set of rational numbers, so we define $I(n)$ to be the maximum 
	value in $\mathcal{I}_n$.
}
\end{definition}
	
With these definitions we can state our motivating general problems.
	
\begin{problem}	  \hfill
	\begin{enumerate}[(i)]
		\item For a given $n$, can we determine $\mathcal{I}_n$?
		\item For a given $n$, can we determine $I(n)$?
		\item If $I(n)$ is larger than $1$, can we determine the structure of the transitive groups $G$ of degree $n$ with $\rho(G)=I(n)$?
	\end{enumerate}	
		\label{prob:main}
\end{problem}


If $G$ is a transitive group of degree $n\ge 2$, then by Jordan's theorem $G$ has a derangement, so $\Gamma_G$ has at least one edge. Since $\Gamma_G$ is vertex transitive, the clique-coclique bound implies that $\alpha(\Gamma_G) \leq |G|/2$ and hence $\rho(G) \leq n/2$.  Moreover, $\rho(G)=n/2$ if and only if $\Gamma_G$ is 
bipartite. The major result in this paper is the surprising fact that the derangement graph for a transitive group can never be bipartite, unless the transitive group has degree $\le 2$.

\begin{theorem}\label{thrm:main1}
	Let $G$ be a transitive permutation group on $\Omega$. If the derangement graph of $G$ is bipartite, 
	then $|\Omega| \le 2$. 
\end{theorem} 

We actually push this a little further and we prove the following.

\begin{theorem}\label{thrm:main2}
Let $G$ be a transitive permutation group on $\Omega$. If $|\Omega|\ge 3$, then the derangement graph of $G$ contains a triangle.
\end{theorem}

Using the clique-coclique bound, this result leads to the following corollary.

\begin{corollary}\label{cor:bound}
	For any $n \geq 3$, we have $I(n)\le \frac{n}{3}$.
\end{corollary}

In Section~\ref{sect:basic}, we present some basic results on the intersection density for transitive groups with a focus on groups with a derangement graph that is a
 complete multipartite graph or the join of several graphs. Section~\ref{sec:main-proof} is dedicated to the proof of Theorem~\ref{thrm:main1}. Section~\ref{sec:triangles} gives the proof of Theorem~\ref{thrm:main2}.
In Section~\ref{sect:counter-example}, we give examples of groups that meet the bound in Corollary~\ref{cor:bound} and other groups that have a derangement graph that is a complete bipartite graph. 
We conclude in Section~\ref{sec:future} with some conjectures and further questions.

\section{Basic Results on Intersection Density} \label{sect:basic}

In this section we state some simple results.

\begin{lemma}
Let $G$ be a permutation group.
\begin{enumerate}
\item Then $\rho(G) \geq 1$.
\item $G$ has the EKR property if and only if $\rho(G) = 1$.
\item If $G$ is $2$-transitive, then $\rho(G) = 1$.
\end{enumerate}
\end{lemma}
\begin{proof}
The first two statements  are immediate. The last statement follows from~\cite{meagher2016erdHos} in which it is proven that every 2-transitive group has the EKR-property.
\end{proof}
We can find $\mathcal{I}_n$ whenever $n$ is a prime number. 

\begin{lemma}
If $G$ is transitive of prime degree $n$, then $\rho(G) =1$ and $\mathcal{I}_n = \{1\}$.
\end{lemma}
\begin{proof}
Let $G$ be transitive of degree $n$, with $n$ a prime number, and let $P$ be a Sylow $n$-subgroup of $G$. Then $P$ is a regular group and hence it is a clique of size $n$ in $\Gamma_G$. Thus, from the clique-coclique bound, we have $\rho(G) = 1$ and $\mathcal{I}_n=\{1\}$.
\end{proof}

Our goal is to find transitive groups that have a large intersection density. We will look for groups that have a large subgroup that is intersecting. We first note that it is simple to check if a subgroup is an intersecting set.

\begin{lemma}
	Let $H$ be a permutation group. Then $H$ is intersecting if and only if it is derangement free. \label{lem:intersecting_subgroups}
\end{lemma}
\begin{proof}
	If $H$ is intersecting, then each $h \in H$ intersects the identity element, and hence  has a fixed point.	
	Conversely, if $H$ is derangement free, then for any $g, h\in H$ the element $gh^{-1}$ is in $H$, so is not a derangement. Thus $g$ and $h$ are intersecting.
\end{proof}

This result can be translated to a statement about the intersection density.

\begin{corollary}
If $G$ is transitive of degree $n$ with a derangement-free subgroup $H$, then $\rho(G) \geq \frac{n}{[G:H]}$.
\end{corollary}

\begin{definition}{\rm
Let $G$ be a transitive permutation group on $\Omega$ and let $\omega\in \Omega$. Since the action of $G$ on $\Omega$ is transitive, we can write the set of elements in $G$ that fix at least one point by 
$\cup_{g\in G}G_\omega^g$. So the set of derangements in $G$ is the set
\begin{equation}\label{eq:second}
\mathcal{D}=G\setminus \bigcup_{g\in G}G_\omega^g.
\end{equation}
Define $H_G$ to be the subgroup generated by the elements of $G$ that fix at least one point, that is,
\begin{align}\label{eq:HG}
H_G &:= \left\langle \bigcup_{g\in G}G_\omega^g \right\rangle=\langle G\setminus\mathcal{D}\rangle.
\end{align}
}
\end{definition}
If $H_G$ is a proper subgroup of $G$, then we can get more information about the structure of $G$. First, we recall the definition for a block of imprimitivity. We say that $\emptyset\ne S\subseteq \Omega$ is a \textbf{\textit{block}}  of $G$ if $S^g = S$ or $S^g\cap S = \varnothing$, for every $g\in G$. Obviously, any subset of $\Omega$ of size $1$ or $|\Omega|$ is a block; we call these \textbf{\textit{trivial blocks}}. We say $G$ is \textbf{\textit{imprimitive}} if it has non-trivial blocks; otherwise, $G$ is called \textbf{\textit{primitive}}. 

\begin{proposition}[{\cite[Proposition~7.1]{wielandt2014finite}}]\label{prop:normals}
Let $N$ be a normal subgroup of the transitive permutation group $G$. Then, the orbits of $N$ are  blocks of $G$.
In particular, if $N$ is intransitive and $N\ne 1$, then $G$ is imprimitive. 
\end{proposition}

We can apply this to $G$ when $H_G$ is a  proper subgroup.

\begin{proposition}\label{prop:HG}
Assume $H_G$, as defined in~\eqref{eq:HG}, is a proper subgroup of $G$, then
\begin{enumerate}
\item  $H_G$ is normal, \label{normality_of_H}
\item $H_G$ is intransitive and has $[G:H]$ orbits on $\Omega$,
\item if $H_G\ne 1$, then  $G$ is imprimitive and the orbits of $H_G$ are blocks for $G$.
\end{enumerate}
\end{proposition}
\begin{proof}
From~\eqref{eq:HG} and from the fact that the set $\bigcup_{g\in G}G_\omega^g$ is left invariant by the action of $G$ on itself by conjugation, we obtain $H_G\unlhd G$. 

Let $\omega\in \Omega$. If $H_G$ is transitive on $\Omega$, then $G=G_\omega H_G$. However, since $G_\omega\le H_G$, we deduce $G=H_G$, contradicting the fact that $H_G$ is a proper subgroup of $G$. The $H_G$-orbit containing $\omega$ has size $[H_G:G_\omega]$ and hence $H_G$ has $|\Omega|/[H_G:G_\omega]=[G:G_\omega]/[H_G:G_\omega]=[G:H_G]$ orbits.
  
The third statement follows from Proposition~\ref{prop:normals}.
\end{proof}


A graph $X$ on $n$ vertices is a \textbf{\textit{join}} of graphs if the vertices of $X$ can be partitioned into parts $\{X_1,X_2,\dots,X_k\}$, where $k\leq n$, so that every vertex in $X_i$ is adjacent to every vertex in $X_j$ for $i\neq j$. The complete $n$-partite graph $K_{\ell,\ell, \dots, \ell}$ is the join of copies of the empty graph $\overline{K_\ell}$. A join is \textbf{\textit{trivial}} if either there is only one part (so $k=1$), or each part has size one (so $k=n$).

\begin{lemma}\label{lem:caycomps}
Let $G$ be a group, let $X$ be an inverse-closed subset of $G\setminus\{1\}$ and let  $\Gamma=\Cay(G,X)$ be the Cayley graph of $G$ with connection set $X$.
The graph $\Gamma$ is a non-trivial join if and only if the set $G \backslash X$ does not generate $G$. Further, if $H := \langle G \backslash X \rangle$, then $\Cay(G, X)$ is the join of $[G:H]$ copies of $\Cay(H,X \cap H)$.
\end{lemma}
\begin{proof}
The graph $\Gamma$ is the join of $k$ graphs if and only if the graph 
$\overline{\Gamma} =\Cay(G, G\backslash (X \cup  \{1\} ) )$ is disconnected with $k$ components. 
As $\overline{\Gamma}$ is a Cayley graph, the number of components in $\overline{\Gamma}$ is equal to $[G:H]$.
Further, the complement of any component of $\overline{\Gamma}$ is isomorphic to $\Cay(H, X \cap H)$.
\end{proof}

\begin{proposition}
	If $\Gamma_G$ is a non-trivial join, then $G$ is imprimitive.\label{prop:join}
\end{proposition}
\begin{proof}
	Assume $\Gamma_G$ is a non-trivial join. Then, by Lemma~\ref{lem:caycomps} applied with $X:=\mathcal{D}$, $H_G$ is a non-trivial and non-identity subgroup of $G$. 
By Proposition~\ref{prop:HG}, $G$ is imprimitive.
\end{proof}

The converse of this proposition does not hold. There are imprimitive groups with a derangement graph that is not a join. We also give some characterizations of the structure of derangement graphs of imprimitive groups with respect to the subgroup $H_G$.

\begin{theorem}
	Let $G$ be a transitive permutation group on $\Omega$. 
	\begin{enumerate}
		\item  If $H_G = \{ 1 \}$, then $\Gamma_G$ is a complete graph and $G$ acts regularly. In particular, $G$ has the EKR property and $\rho(G) = 1$.
		\item If $H_G$ is a proper subgroup, then $\Gamma_G$ is a join of $[G:H_G]$ copies of $\Gamma_{H_G}$. Further $\alpha(\Gamma_G)= \alpha(\Gamma_{H_G})$ and $\rho(G) = \rho(H_G)$.
		\item If $H_G = G$, then $\Gamma_G$ is not a join.
	\end{enumerate}\label{theorem:main3}
\end{theorem}
\begin{proof}
If $H_G =\{ 1\}$, then every non-identity element of $G$ is a derangement and the first statement follows.

The first part of the second statement follows from Lemma~\ref{lem:caycomps} applied with $X:=\mathcal{D}$. In particular, 
$\Gamma_G$ is the join of $[G:H_G]$ copies of $\Gamma_{H_G}$ and hence
 $\alpha(\Gamma_G)= \alpha(\Gamma_{H_G})$.  As $H_G$ is intransitive with $[G:H_G]$ orbits, every $H_G$-orbit has size $\frac{|\Omega|}{[G:H_G]}$. 	
Using the definition of intersection density, we have
\[
\rho(G) = \frac{|\Omega| \, \alpha(\Gamma_{G})}{|G|} 
            = \frac{|\Omega| \, \alpha(\Gamma_{H_G})}{[G:H_G] \, |H_G|} 
            = \rho( H_G).
\]

Finally, if $H_G =G$, then $\overline{\Gamma_G} = \Cay(G, \bigcup_{g\in G}G_\omega^g )$ is connected, so $\Gamma_G$ is not a join.
\end{proof}

We will focus on derangement graphs that are the join of empty graphs, these are exactly the complete multipartite graphs.

\begin{lemma}
	Let $G$ be transitive. Then $\Gamma_G$ is a complete multipartite graph if and only if 
	$H_G$ is a derangement-free non-identity subgroup of $G$.\label{lem:complete_multi}
\end{lemma}
\begin{proof}
Assume $\Gamma_G$ is the complete $k$-partite graph $K_{\ell ,\ell ,\dots , \ell}$ (with $\ell, k \geq 2$).
The graph $\overline{\Gamma_G} = \Cay(G, \cup_{g \in G} G_\omega^g)$ is the union of $k$ copies of $K_{\ell}$. The vertices in the copy of $K_\ell$ that contains the identity are the elements of the subgroup generated by $\cup_{g \in G} G_\omega^g$, so $H_G$. This implies $H_G$ is derangement free and $|H_G|=\ell>1$.

Conversely, if $H_G$ is derangement free and non-trivial, then each connected component of $\overline{\Gamma_G}$ is a complete graph. Finally, $H_G$ is a proper subgroup of $G$ and there are $[G:H]$ components of $\overline{\Gamma_G}$. This implies that $\Gamma_G$ is isomorphic to the $[G:H]$-partite graph $K_{|H_G|,\dots, |H_G|}$.
\end{proof}

\begin{corollary}
If $H_G$ is a proper subgroup of $G$ and $H_G$ is derangement free, then $\Gamma_G$ is a complete multipartite graph with $[G:H_G]$ parts and $\rho(G) = \frac{n}{[G:H_G]}$. 
\end{corollary}


\section{Bipartite Derangement Graphs}\label{sec:main-proof}

In this section we prove Theorem~\ref{thrm:main1}.
The proof of  Theorem~\ref{thrm:main1} is quite involved and it occupies this whole section. We divide the proof in various subsections, where in each subsection we refine our understanding of the structure of a minimal counterexample $G$ to Theorem~\ref{thrm:main1}. Then, the information in one subsection is used in the subsequent subsections in order to obtain further information, until we reach a contradiction. 

In our inductive argument, we first prove that $G$ is a \textbf{\textit{biprimitive group}} on $\Omega$, that is, $G$ admits a system of imprimitivity consisting of two blocks and the stabilizer of this system of imprimitivity acts primitively on each block. (More details are given in due course.) Then, we prove that $G$ is almost simple and finally, invoking the Classification of the Finite Simple Groups, we complete our proof. Most of our reduction uses the ideas in the work of Garonzi and  Lucchini~\cite{Lucchini}. Indeed, the authors in~\cite{Lucchini} are interested in a group-theoretic question very much related to Theorem~\ref{thrm:main1}. (Again, more details are given in due course.) Unfortunately, we were not able to use the results in~\cite{Lucchini} directly to our problem and hence we had adapted the arguments in~\cite[Section~3]{Lucchini} to our current needs.

\subsection{General notation and background results}\label{bsec:1}

\begin{definition}\label{def:1}{\rm
Let $k$ be a positive integer and let $G$ be a finite non-cyclic group. A \textbf{\textit{normal $k$-covering}} of $G$ is a family $H_1,\ldots, H_k$ of $k$ distinct proper subgroups of $G$ with the property that every element of $G$ is contained in $H_i^g$,
for some $i \in \{1, \ldots , k\}$ and for some $g \in G$, that is,
$$G=\bigcup_{i=1}^k\bigcup_{g\in G}H_i^g.$$
Clearly, if $G$ is a cyclic group, then $G$ admits no normal $k$-covering, because the generators of $G$ lie in no proper subgroups.

The \textbf{\textit{normal covering number}} of the group $G$, denoted by $\gamma(G)$, is the smallest integer
$k$ such that $G$ admits a normal $k$-covering. If $G$ is
cyclic, we set $\gamma(G) = \infty$, with the convention that $k < \infty$ for every integer $k$.}
\end{definition}

It follows from a theorem of Jordan that $\gamma(G)\ge 2$, for every finite group $G$. Indeed, if $\gamma(G)=1$, then $G$ admits a proper subgroup $H$ with the property that $$G=\bigcup_{g\in G}H^g.$$
There is another way to phrase this equality: the group $G$ acting on the right cosets of $H$ in $G$ admits no derangement. However, this contradicts a celebrated theorem of Jordan; see, for instance, the beautiful expository article of Serre~\cite{Serre}. 

The following two lemmas are preparatory for the proof of Theorem~\ref{thrm:main1}: these are Lemmas~9 and~10 in~\cite{Lucchini} (we include a proof here for completeness). 
\begin{lemma}\label{l:prel1}
Let $Y$ be a proper subgroup of a finite group $X$ and let $Z\unlhd X$ with $X=YZ$. Then $$Z\ne \bigcup_{x\in X}(Y\cap Z)^x.$$
\end{lemma}
\begin{proof}Suppose $Z=\bigcup_{x\in X}(Y\cap Z)^x.$ Since $Z\unlhd X$, we have $Y\cap Z\unlhd Y$ and hence $(Y\cap Z)^y=Y\cap Z$, for every $y\in Y$.  Hence
\begin{align*}
Z&=\bigcup_{x\in X}(Y\cap Z)^x=\bigcup_{z\in Z}\bigcup_{y\in Y}(Y\cap Z)^{yz}=\bigcup_{z\in Z}(Y\cap Z)^z.
\end{align*}
Therefore,  $Y\cap Z$ is a subgroup of $Z$ whose $Z$-conjugates cover the whole of $Z$. From the theorem of Jordan, this implies $Z=Y\cap Z$, that is $Z \le Y$. Thus $X=YZ=Y$, contradicting the fact that $Y$ is a proper subgroup of $X$.
\end{proof}

\begin{lemma}\label{l:prel2}
Let $Y$ be a proper subgroup of a finite group $X$ and let $Z_1$ and $Z_2$ be two distinct minimal normal subgroups of $X$. If 
$X=YZ_1=YZ_2$, then $Y\cap Z_1=Y\cap Z_2=1$.
\end{lemma}
\begin{proof}
Assume $X=YZ_1=YZ_2$. As $Z_1\unlhd X$, we have $$Y\cap Z_1\unlhd Y.$$ Since $Z_1$ and $Z_2$ are distinct minimal normal subgroups of $X$ and since $Z_1\cap Z_2$ is normal in $X$, we must have $Z_1\cap Z_2=1$. From this, it immediately follows that the elements of $Z_1$ and $Z_2$ commute with each other. Therefore $Y\cap Z_1$ is centralized by $Z_2$.

From the previous paragraph, $Y\cap Z_1$ is normalized by $Y$ and centralized by $Z_2$ and hence $Y\cap Z_1$ is normalized 
by $\langle Y,Z_2\rangle=YZ_2=X$, that is, $Y\cap Z_1\unlhd X$. Since $Z_1$ is a minimal normal subgroup of $X$, we deduce $Z_1\le Y$ or $Y\cap Z_1=1$. If $Z_1\le Y$, then $X=YZ_1=Y$, contradicting the fact that $Y$ is a proper subgroup of $X$. Thus $Y\cap Z_1=1$.

A similar argument yields $Y\cap Z_2=1$. 
\end{proof}

\subsection{Preliminary reductions}\label{sec:1} 

We let $G$ be a transitive permutation group on $\Omega$ such that the derangement graph $\Gamma_G$ of $G$ is bipartite. We fix a bipartition $$\mathcal{B}=\{H,G\setminus H\}$$ 
of the vertices of $\Gamma_G$. Without loss of generality suppose that $H$ contains the identity element of $G$. Since $\Gamma_G$ is bipartite, this implies that no elements of $H$ are derangements. 

The group $G$ acts as a group of automorphisms on the graph $\Gamma_G$ via its right regular representation. Since $G$ acts transitively on the vertices of $\Gamma_G$, the subgroup $G_\mathcal{B}$ of $G$ fixing setwise the two parts of the bipartition $H$ and $G\setminus H$ has index $2$ in $G$ and acts transitively on both $H$ and $G\setminus H$.  As $1\in H$, we deduce
$$H=\{1^{x}\mid x\in G_\mathcal{B} \}=\{1\cdot x\mid x\in G_\mathcal{B}\}=G_\mathcal{B}.$$
This shows that $H$ is a subgroup of $G$ with 
\begin{equation}\label{eq:-1}
[G:H]=2\hbox{ and }H\unlhd G.
\end{equation}

As usual, let $\mathcal{D}$ be the set of derangements of $G$. The subgroup $\langle\mathcal{D}\rangle$ of $G$ generated by $\mathcal{D}$ contains $G\setminus H$ and also the identity element of $G$. Therefore, $|\langle \mathcal{D}\rangle| \ge  |G|/2+1$. This shows that
\begin{equation}\label{eq:first}G=\langle\mathcal{D}\rangle.
\end{equation}
From~\eqref{eq:first}, we deduce that $\Gamma_G$ is connected.

Since $H\unlhd G$ by~\eqref{eq:-1},  $G_\omega H$ is subgroup of $G$, for every $\omega\in \Omega$. As $H$ has no derangements, by the theorem of Jordan, $H$ cannot be transitive on $\Omega$. Thus $H$ is intransitive on $\Omega$ and $G_\omega H$ is a proper subgroup of $G$. However, $H\le G_\omega H<G$ and $[G:H]=2$; therefore $H=HG_\omega$ and $G_\omega\le H$. This yields
\begin{equation}\label{eq:-2}
G_\omega \le H, \quad   \forall \omega \in \Omega.
\end{equation}
In particular, $G_\omega=H_\omega$, for all $\omega\in \Omega$.

From~\eqref{eq:-2} and from the fact that $H$ is intersecting, we deduce
\begin{equation}\label{eq:00}
H=\bigcup_{\omega \in \Omega}G_\omega = \bigcup_{\omega \in \Omega}H_\omega.
\end{equation}

For the rest of the proof we fix $\omega\in \Omega$ and $g'\in G\setminus H$ and we set $\omega':=\omega^{g'}$.

As $[G:H]=2$ and as $H$ is intransitive on $\Omega$, we deduce that $H$ has two orbits on $\Omega$; namely 
\begin{align}\label{eq:DeltaDefined}
\Delta:=\omega^H=\{\omega^h\mid h\in H\} 
\mbox{ and } 
\Delta':=\omega'^H=\{\omega'^h\mid h\in H\}.  
\end{align}
From~\eqref{eq:00} and~\eqref{eq:DeltaDefined}, we deduce
$$H=\bigcup_{\delta\in \Delta}H_\delta\cup\bigcup_{\delta'\in\Delta'}H_{\delta'}$$
and hence
\begin{equation}\label{eq:2}
H=\bigcup_{h\in H}H_\omega^h\cup\bigcup_{h\in H}H_{\omega'}^h.
\end{equation}
In other words, $H$ has two subgroups $H_\omega$ and $H_{\omega'}$ such that the conjugates of $H_\omega$ and $H_{\omega'}$ cover the whole of $H$.

Suppose now that $H_\omega=H$ (or that $H_{\omega'}=H$). 
As $H_\omega=G_\omega$, we deduce $G_\omega=H$. 
Since $G$ is a transitive permutation on $\Omega$, $G_\omega$ is a core-free subgroup of $G$, that is, the only normal subgroup of $G$ contained in $G_\omega$ is the identity subgroup. However, from~\eqref{eq:-1}, we have $G_\omega=H\unlhd G$ and hence $H=G_\omega=1$. This gives $|G|=[G:H]=2$ and hence $|\Omega|=2$. Therefore, for the rest of the proof we may suppose that $|\Omega|>2$ and hence, in particular, 
\begin{equation}\label{eq:-3}H_\omega\,\textrm{ and }\,H_{\omega'}\textrm{ are proper subgroups of }H.
\end{equation}
With the terminology in Definition~\ref{def:1}, from~\eqref{eq:2}, we have
\begin{equation}\label{eq:3}\gamma(H)=2,
\end{equation} that is, the normal covering number of $H$ is $2$.

 In view of the conclusion in the statement of Theorem~\ref{thrm:main1}, our task here is to reach a contradiction. Among all possible transitive permutation groups $G$ with $|\Omega|>2$ and with the derangement graph of $G$ bipartite, choose $G$ with $|\Omega|+|G|$ as small as possible.


\subsection{The action of $H$ on $\Delta$ and $\Delta'$}

The goal of this subsection is to prove that $H$ acts primitively and faithfully on both $\Delta$ and $\Delta'$ (defined in~\eqref{eq:DeltaDefined}).

Let $M$ be a subgroup of $H$ with $$H_\omega \le M<H.$$ As $H_\omega=G_\omega$, we have $G_\omega\le M$ and hence the $M$-orbit $\omega^M$ is a block of imprimitivity for the action of $G$ on $\Omega$ contained in $\Delta$, see~\cite[Theorem~$1.5$A]{DM}. Let $\Omega'$ be the system of imprimitivity determined by the block $\omega^M$ and let $G'$ be the permutation group induced by the action of $G$ on $\Omega'$. Since the derangement graph for the action of $G$ on $\Omega$ is bipartite, so is the derangement graph for the action of $G'$ on $\Omega'$. Now, 
\[
|\Omega'|=[G:M]=\frac{[G:H_\omega]}{[M:H_\omega]}=\frac{[G:H]\, [H:H_\omega]}{[M:H_\omega]}=2\frac{[ H:H_\omega]}{ [M:H_\omega]}>2
\]
and hence, from our minimal choice of $G$, we must have that $\Omega'=\Omega$ and $G'=G$. Clearly, this is only possible if $M=H_\omega$. Since $M$ was an arbitrary proper subgroup of $H$ containing 
$H_\omega$, we deduce that
\begin{equation}\label{eq:-4}
H_\omega\textrm{ and }H_{\omega'} \textrm{ are maximal subgroups of }H.
\end{equation} 
There is another way to say this: the action of $H$ on $\Delta$ and on $\Delta'$ is primitive, see~\cite[Corollary~$1.5$A]{DM}.

Let $H_{(\Delta)}$ be the subgroup of $H$ fixing pointwise each element of $\Delta$ and, similarly, let $H_{(\Delta')}$ be the subgroup of $H$ fixing pointwise each element of $\Delta'$. Now, $H_{(\Delta)}\cap H_{(\Delta')}$ consists of permutations in $H$ fixing each element of $\Delta\cup \Delta'=\Omega$. Since $H$ acts faithfully on $\Omega$, we have $$H_{(\Delta)}\cap H_{(\Delta')}=1.$$ Suppose $H_{(\Delta)}\ne 1\ne H_{(\Delta')}$. As $\Delta$ is an $H$-orbit, $H_{(\Delta)}\unlhd H$. From~\eqref{eq:-4}, we have $H_{(\Delta)}\le H_{\omega'}$ or $H=H_{\omega'}H_{(\Delta)}$. However, if $H_{(\Delta)}\le H_{\omega'}$, then for every $h\in H$ we have $$H_{(\Delta)}=(H_{(\Delta)})^h\le (H_{\omega'})^h=H_{\omega'^h}.$$Hence $$H_{(\Delta)} \le \bigcap_{\delta'\in \Delta}H_{\delta'}=H_{(\Delta')},$$
which yields $1=H_{(\Delta)}\cap H_{(\Delta')}=H_{(\Delta)}\ne 1$, a contradiction. Thus $H=H_{\omega'}H_{(\Delta)}$. In particular, we are in the position to apply Lemma~\ref{l:prel1} with $(X,Y,Z)=(H,H_{\omega'},H_{(\Delta)})$. We deduce that there exists $$y\in H_{(\Delta)}\setminus\bigcup_{h\in H}H_{\omega'}^h.$$ Applying this same argument with the roles of $H_{(\Delta)}$ and $H_{(\Delta')}$ interchanged, we deduce that there exists $$z\in H_{(\Delta')}\setminus\bigcup_{h\in H}H_\omega^h.$$
We now consider the element $x:=yz$. This would be a derangement in the group $H$, which is a contradiction. This contradiction has arisen from assuming $H_{(\Delta)}\ne 1\ne H_{(\Delta')}$ and hence (since $H_{(\Delta)}$ and $H_{(\Delta')}$ are conjugate in $G$) we have $H_{(\Delta)}=H_{(\Delta')}=1$. Summing up,
\begin{equation}\label{eq:third}
H \textrm{ acts primitively and faithfully on both }\Delta \textrm{ and }\Delta'.
\end{equation}

We conclude this section recalling some properties of primitive groups. The subgroup of an abstract group $X$ generated by its minimal normal subgroups is called the \textbf{\textit{socle}}. We let $M$ be the socle of $H$. Since $H$ is a primitive group (on either $\Delta$ or $\Delta'$), $M$ is either a minimal normal subgroup of $H$ or $M$ is the direct product of two distinct isomorphic minimal normal subgroups of $H$, see~\cite[Theorem~4.3B]{DM}. In the next subsection, we show that only the first case is possible.

\subsection{The group $H$ has a unique minimal normal subgroup}

From the previous subsection we know that $M=N_1$ or $M=N_1 \times N_2$ where $N_1,N_2$ are minimal normal subgroups of $H$. In this subsection we show that the second case cannot hold.

Assume $M=N_1 \times N_2$. Let $K\in \{H_\omega,H_{\omega'}\}$. By~\eqref{eq:-4}, $K$ is a maximal subgroup of $H$; as $N_i\unlhd H$, we deduce $H=KN_i$ or $N_i\le K$ for $i\in \{1,2\}$. However, the last possibility contradicts~\eqref{eq:third} because $H$ acts faithfully on both $\Delta$ and $\Delta'$. 
Thus $$H_\omega N_1=H_\omega N_2=H=H_{\omega'}N_1=H_{\omega'}N_2.$$

From Lemma~\ref{l:prel2}, we deduce $$H_\omega \cap N_1=H_{\omega}\cap N_2=1=H_{\omega'}\cap N_1=H_{\omega'}\cap N_2.$$
In particular, for every $h\in H$, we have $1=(H_\omega \cap N_1)^h=H_{\omega}^h\cap N_1$ and $1=(H_{\omega'}\cap N_1)^h=H_{\omega'}^h\cap N_1$. Therefore,
\[
1=N_1\cap \left(\bigcup_{h\in H} H_\omega^h  \cup   \bigcup_{h\in H}H_{\omega'}^h\right)
  =N_1\cap H=N_1,
\]
contradicting the fact that $N_1\ne 1$. This shows that $M=N_1 \times N_2$ is not possible, so $M=N_1$, and therefore, $M$ is a minimal normal subgroup of $H$. 

Recall (see also~\cite[Chapter~$4$]{DM}) that a minimal normal subgroup of a group is the direct product of pairwise isomorphic simple groups. In particular, we may write
$$M=S_1\times \cdots\times S_r,$$
for some positive integer $r$ and for some simple groups $S_1,\ldots,S_r$ with $S_1\cong S_2\cong\cdots \cong S_r$. When $S_1$ is abelian, we deduce that $S_1$ has prime order $p$ and hence $M$ is an elementary abelian $p$-group. When $S_1$ is non-abelian, it is elementary to verify that $S_1,\ldots,S_r$ are the only minimal normal subgroups of $M$; moreover, the fact that $M$ is a minimal normal subgroup of $H$ implies that the action of $H$ by conjugation on
$\{S_1,\ldots,S_r\}$ is transitive.


\subsection{Preliminary observations on the structure of $M$}  

The maximality of $H_\omega$ and $H_{\omega'}$ in $H$ (see~\eqref{eq:-4}) and the fact that $H$ acts 
faithfully on both $\Delta$ and $\Delta'$ (a.k.a.~\eqref{eq:third}) yield
\begin{equation}\label{eq:fifth}
H=H_\omega M=H_{\omega'}M.
\end{equation}

Intersecting the two members of the equality in~\eqref{eq:2} with $M$, we obtain
\[
M= \left( \bigcup_{h\in H}(M\cap H_\omega)^h \right)  \bigcup \left( \bigcup_{h\in H}(M\cap H_{\omega'})^h \right).
\]
Since  $H=H_\omega M=H_{\omega'}M$, we deduce
\begin{align*}
M&=       \left( \bigcup_{x\in H_\omega}   \bigcup_{h\in M}  (M\cap H_\omega)^{xh} \right) 
 \bigcup  \left( \bigcup_{x\in H_{\omega'}}\bigcup_{h\in M}  (M\cap H_{\omega'})^{xh} \right).
\end{align*}
Observe now that, as $M\unlhd H$, we have $M\cap H_\omega\unlhd H_\omega$ and $M\cap H_{\omega'}\unlhd H_{\omega'}$, and therefore,
\begin{equation}\label{eq:new}
M=\left( \bigcup_{h\in M}(M\cap H_\omega)^h \right) \bigcup \left( \bigcup_{h\in M}(M\cap H_{\omega'})^h \right) .
\end{equation}

Assume $M\cap H_\omega=1$ or $M\cap H_{\omega'}=1$. Without loss of generality, we may suppose that $M\cap H_\omega=1$. Recall that $g'$ was defined in Section~\ref{sec:1} so that $\omega^{g'}=\omega'$. Since $M$ is the unique minimal normal subgroup of $H$ and $H\unlhd G$, $M$ is normalized by $g'$ and hence 
\[
1=(M\cap H_\omega)^{g'}=M\cap H_\omega^{g'}=M\cap H_{\omega^{g'}}=M\cap H_{\omega'}.
\]
However, this and~\eqref{eq:new} yield $M=1$, which is a contradiction. Therefore $$M\cap H_\omega\ne 1 \ne M\cap H_{\omega'}.$$

This result will be used to show that $S_1$ is not abelian, so we will assume $S_1$ is abelian and derive a contradiction. If $S_1$ is abelian, then $M$ is also abelian and hence $M\cap H_{\omega}\unlhd M$. Since $M\unlhd H$, we also have $M\cap H_\omega\unlhd H_\omega$ and hence $M\cap H_\omega \unlhd \langle H_\omega,M\rangle=H_\omega M=H$, by~\eqref{eq:fifth}. Since $M$ is a minimal normal subgroup of $H$, we have either $M\cap H_\omega=1$ or $M\le H_\omega$. However, the first possibility contradicts the previous paragraph and the second possibility contradicts~\eqref{eq:third}. Therefore
\begin{equation}\label{eq:sixth}
S_1\hbox{ is a non-abelian simple group}.
\end{equation}

\subsection{The group $M$ is a non-abelian simple group.} 
In this subsection we show that $M$ is a non-abelian simple group. We argue by contradiction and suppose that $M$ is not a non-abelian simple group, that is,
$$r>1.$$ 

We have $M = S_1 \times \dots \times S_r$ and, for simplicity, we let $S$ be a non-abelian simple group with $S\cong S_j$, for all $j\in \{1,\ldots,r\}$. Further, let $$\pi_1:M\to S_1$$ be the projection of $M$ onto the first component $S_1$. 

We now use the fact that $H$ acts faithfully and primitively on $\Delta$ and on $\Delta'$ via the O'Nan-Scott theorem~\cite[Chapter~$4$]{DM}. This theorem gives a satisfactory description of the embedding of $M$ in $H$ and of the intersection of the stabilizer of a point with $M$. We need the following information from the O'Nan-Scott classification of primitive groups: The maximal subgroups $X$ of $H$ with $XM=H$ and $M\cap X \ne 1$ are of the following two types:

\begin{enumerate}
\item \textbf{Product Type:}
When $1<\pi_1(M\cap X)<S$; in this case, 
$$M\cap X=T_1\times T_2\times \cdots \times T_r,$$
with $1<T_i<S_i$ and $T_i\cong T_j$, for every $i,j\in \{1,\ldots,r\}$;

\item \textbf{Diagonal Type:} when $S=\pi_1(M\cap X)$; in this case, there exists a partition $\Phi$ 
of $\{1,\ldots,r\}$ such that 
$$M\cap X=\prod_{B\in \Phi}D_B,$$
where all the blocks $B=\{j_1,\ldots,j_\ell\}$ have the same cardinality $\ell>1$ and, for every $B\in \Phi$, $D_B$ is a full diagonal subgroup of 
$$\prod_{j\in B}S_j,$$
that is, for every $k\in \{1,\ldots,\ell\}$, there exists $\phi_{j_k}\in \mathrm{Aut}(S_{j_k})$ such that
$$D_B:=\{(x^{\phi_{j_1}},x^{\phi_{j_2}},\ldots,x^{\phi_{j_\ell}})\mid x\in S\}\le S_{j_1}\times S_{j_2}\times \cdots \times S_{j_\ell}.$$
\end{enumerate}
  
In particular, we may apply these considerations twice: with $X:=H_\omega$ and with $X:=H_{\omega'}$. Therefore, replacing the role of $\omega$ and $\omega'$ if necessary, we have three possibilities:
\begin{enumerate}
\item\label{case1} $H_\omega$ and $H_{\omega'}$ are both of diagonal type;
\item\label{case2} $H_\omega$ is of product  type and $H_{\omega'}$ is of diagonal type;
\item\label{case3} $H_\omega$ and $H_{\omega'}$ are both of product type.
\end{enumerate}

We deal with these three cases in turn. Assume Case~\eqref{case1}. Let $$\Lambda:=\{(s,\underbrace{1,\ldots,1}_{r-1\textrm{ times}})\mid s\in S \} \subseteq M.$$ 
By the way in which maximal subgroups of diagonal type are defined, the blocks of the partition giving rise to the diagonal subgroup have cardinality $\ell>1$. From this it  follows that   
\[
\Lambda\cap H_\omega^h=\Lambda\cap H_{\omega'}^h=1,
\]
for every $h\in M$, contradicting~\eqref{eq:new}. 

Assume Case~\eqref{case2}. We have $H_\omega\cap M=T_1\times \cdots \times T_r$, with $T_i\cong T_j$  and with $T_i< S_i$. Let $T$ be a subgroup of $S$ with $T\cong T_i$, for each $i\in\{1,\ldots,r\}$. As $T$ is a proper subgroup of $S$, from Jordan's theorem, there exists $s\in S\setminus \bigcup_{x\in S}T^x$. Consider $$m:=(s,\underbrace{1,\ldots,1}_{r-1\textrm{ times}})\in M.$$ Then $$m\notin\bigcup_{h\in M}(M\cap H_{\omega'})^h,$$ from the description of the elements in diagonal subgroups (again, the blocks of the partition giving rise to the diagonal subgroup have cardinality $\ell>1$). Also $$m\notin\bigcup_{h\in M}(M\cap H_\omega)^h,$$ from our choice of $s$. However, this contradicts~\eqref{eq:new}. 

Assume Case~\eqref{case3}. Let 
\begin{align*}
M\cap H_\omega&:=T_1\times \cdots \times T_r,\\
M\cap H_{\omega'}&:=U_1\times \cdots \times U_r,
\end{align*}
with $T_i\cong T_j$, $U_i\cong U_j$, $T_i< S_i$ and $U_i< S_i$, for every $i,j\in \{1,\ldots,r\}$. (Recall that in this subsection we are arguing by contradiction and we are assuming $r>1$.) Since $T_1$ and $U_2$ are proper subgroups of $S$, from Jordan's theorem, there exists 
$$a\in S\setminus\bigcup_{s\in S}T_1^s \hbox{ and }b\in S\setminus\bigcup_{s\in S}U_2^s.$$ Consider $$m:=(a,b,\underbrace{1,\ldots,1}_{r-2\textrm{ times}})\in M.$$ Then $$m\notin\bigcup_{h\in M}(M\cap H_{\omega})^h \cup \bigcup_{h\in M}(M\cap H_{\omega'})^h,$$ contradicting 
again~\eqref{eq:new}. Therefore $r=1$, that is, $M$ is a non-abelian simple group.

\subsection{Conclusive analysis}

From~\eqref{eq:new}, we have $\gamma(M)=2$ and, from the previous subsection, $M$ is a non-abelian simple group. At this point, we could refer directly to the classification of the simple groups admitting a $2$-covering, which in turn relies on the Classification of the Finite Simple Groups. This classification is spread through various papers. For instance,~\cite{bubboloni} deals with alternating groups;~\cite{pellegrini} deals with sporadic simple groups and exceptional groups of Lie type;~\cite{bubboloni1,bubboloni2} deal with simple classical groups.  However, first we obtain another reduction based on  the fact that the normal $2$-covering arising in~\eqref{eq:new} is rather special.

Recall that $g'\in G$ and $\omega'=\omega^{g'}$. Since $M$ is a characteristic subgroup of $H$ and $H\unlhd G$, we deduce $$(M\cap H_\omega)^{g'}=M^{g'}\cap H_{\omega}^{g'}=M\cap H_{\omega'}.$$
In particular, $M\cap H_\omega$ and $M\cap H_{\omega'}$ are proper subgroups of $M$ conjugate via an automorphism of $M$.

Let $\pi(M)$ be the set of prime numbers dividing the order of $M$ and let $\pi(H_\omega\cap M)$ be the set of prime numbers dividing the order of $H_\omega\cap M$. From the above equality, along with~\eqref{eq:new}, we deduce 
\begin{equation}\label{pi}\pi(M)=\pi(H_\omega\cap M).\end{equation}

For the benefit of the reader we now report~\cite[Corollary~5]{LPS}, tailored to our current notation.
\begin{lemma}\label{l:LPS}Let $M$ be a non-abelian simple group and let $X$ be a proper subgroup of $M$. If $\pi(X)=\pi(M)$, then $(M,X)$ is given in 
Table~$10.7$ in~\cite{LPS}.
\end{lemma}

As $\pi(M\cap H_\omega)=\pi(M)=\pi(M\cap H_{\omega'})$, we are in the position to apply Lemma~\ref{l:LPS} with $X:=M\cap H_\omega$ and with $X:=M\cap H_{\omega'}$. We deduce that $(M,M\cap H_\omega)$ and $(M,M\cap H_{\omega'})$ are in Table~$10.7$ of~\cite{LPS}. We now consider each row in~\cite[Table~$10.7$]{LPS} in turn.

When $(M,M\cap H_\omega)$ is not as in line~1,~3,~4,~5,~6, 
the proof follows with a simple computation with the computer algebra system \texttt{magma}~\cite{magma}. There is no normal $2$-covering of $M$ determined by two proper subgroups conjugate in $\Aut(M)$. For simplicity, we have reported in Table~\ref{table1} the remaining lines of~\cite[Table~10.7]{LPS}.

\begin{table}[!h]
\centering
\begin{tabular}{|c|c|c|c|}
\hline \hline
Line&$M$ & $X\in \{M\cap H_\omega,M\cap H_{\omega'}\}$ & Conditions  \\\hline
1&$\mathrm{Alt}(c)$ &$\Alt(k)\unlhd X\le \Sym(k)\times \mathrm{Sym}(c-k)$& if $p$ is prime and \\
 & & & $p\le c$, then $p\le k$\\ \hline
3&$\mathrm{PSp}_{2m}(q)$&$\nor M{\Omega_{2m}^-(q)}$&$m$ and $q$ even\\ \hline
4&$\mathrm{P}\Omega_{2m+1}(q)$&$\nor M{\Omega_{2m}^-(q)}$&$m$ even and $q$ odd\\ \hline
5&$\mathrm{P}\Omega^+_{2m}(q)$&$\nor M {\Omega_{2m-1}(q)}$&$m$ even\\ \hline
6&$\mathrm{PSp}_{4}(q)$&$\nor M{\mathrm{PSp}_2(q^2)}$&\\\hline
\end{tabular}
\caption{$\nor M{X}$ denotes the normalizer in $M$ of $X$ } \label{table1}
\end{table}

Suppose $(M,M\cap H_\omega)$ is as in line~1. Then $M$ is an alternating group admitting a normal $2$-covering and, from~\cite{bubboloni}, we deduce that $c\le 8$. It can be checked directly, or with the help of a computer, that there are no normal $2$-coverings of $\mathrm{Alt}(c)$ (with $c\le 8$) using two isomorphic subgroups.

Suppose $(M,M\cap H_\omega)$ is as in line~3. We can postpone the case $m=2$ when we deal with line~6. From~\cite[Main Theorem]{bubboloni2}, we deduce that $m=4$, because when $m\ge 5$ the group $\mathrm{PSp}_{2m}(q)$ admits no normal $2$-covering with two isomorphic maximal subgroups.  Now,~\cite[Table~$8.48$]{bhr} lists all the maximal subgroups of $\mathrm{PSp}_8(q)$. From the ``$c$ column'' in~\cite[Table~$8.48$]{bhr}, we deduce that $M\cap H_\omega$ and $M\cap H_{\omega'}$ are conjugate in $M$ (see~\cite[page 374]{bhr} for the definition of $c$). Now,~\eqref{eq:new} contradicts Jordan's theorem.

Suppose $(M,M\cap H_\omega)$ is as in line~4. When $m=2$, we have $\mathrm{P}\Omega_{5}(q)\cong\mathrm{PSp}_4(q)$ and hence we can postpone again this case when we deal with line~6. From~\cite[Main Theorem]{bubboloni2}, we deduce that $m=4$, because when $m\ge 5$ the group $\mathrm{P}\Omega_{2m+1}(q)$ admits no normal $2$-covering.  Now,~\cite[Table~$8.58$]{bhr} lists all the maximal subgroups of $\mathrm{P}\Omega_9(q)$. From the ``$c$ column'' in~\cite[Table~$8.58$]{bhr}, we deduce that $M\cap H_\omega$ and $M\cap H_{\omega'}$ are conjugate in $M$. Now,~\eqref{eq:new} contradicts Jordan's theorem.

Suppose $(M,M\cap H_\omega)$ is as in line~5. Here $m\ge 4$ because $\mathrm{P}\Omega_4^+(q)\cong\mathrm{PSL}_2(q)\times\mathrm{PSL}_2(q)$ is not a non-abelian simple group. From~\cite[Main Theorem]{bubboloni2}, we deduce that $m=4$, because when $m\ge 5$ the group $\mathrm{P}\Omega_{2m}^+(q)$ admits no normal $2$-covering. Now,~\cite[Table~$8.50$]{bhr} lists all the maximal subgroups of $\mathrm{P}\Omega_8^+(q)$. From this list, we cannot deduce that $M\cap H_\omega$ and $M\cap H_{\omega'}$ are conjugate in $M$ and hence we cannot argue as in the previous two cases. Thus, let $V=\mathbb{F}_q^8$ be the $8$-dimensional vector space over the field $\mathbb{F}_q$ with $q$ elements,  and let $\mathfrak{q}:V\to\mathbb{F}_q$ be the hyperbolic non-degenerate quadratic form on $V$ preserved by the covering group $\Omega_8^+(q)$. We may choose a hyperbolic basis $(e_1,e_2,e_3,e_4,f_1,f_2,f_3,f_4)$ for $V$ so that the matrix of the quadratic form $\mathfrak{q}$ with respect to this basis is
\[
\begin{pmatrix}
0&I\\
0&0
\end{pmatrix},
\]
where $0$ and $I$ represent the $4\times 4$ zero and identity matrix. From the ``$c$ column'' in~\cite[Table~$8.50$]{bhr}, we see that when $q$ is even, there are three $\mathrm{P}\Omega_8^+(q)$-conjugacy classes of maximal subgroups of the form $\nor M{\Omega_{7}(q)}=\nor M{\mathrm{Sp}_6(q)}$: one in the Aschbacher class $\mathcal{C}_1$ and two in the Aschbacher class $\mathcal{S}$. Similarly, when $q$ is odd, there are six $\mathrm{P}\Omega_8^+(q)$-conjugacy classes of maximal subgroups of the form $\nor M{\Omega_{7}(q)}$: two in the Aschbacher class $\mathcal{C}_1$ and four in the Aschbacher class $\mathcal{S}$. The maximal subgroups in the Aschbacher class $\mathcal{C}_1$ arise as stabilizers of $1$-dimensional non-degenerate subspaces of $V$, with respect to the form $\mathfrak{q}$. Whereas, the maximal subgroups in the Aschbacher class $\mathcal{S}$ arise via the spin representations of $\Omega_{7}(q)$ as described in~\cite[Section~$5.4$]{kl}. From the description of the conjugacy classes in $\Omega_8^+(q)$ in~\cite{wall}, we see that $\Omega_8^+(q)$ contains the unipotent matrix
\[
\tilde{g}=\begin{pmatrix}
1&1&0&0&0&0&0&0\\
0&1&1&0&0&0&0&0\\
0&0&1&0&0&0&0&0\\
0&0&0&1&1&0&0&0\\
0&0&0&0&1&1&0&0\\
0&0&0&0&0&1&1&0\\
0&0&0&0&0&0&1&1\\
0&0&0&0&0&0&0&1
\end{pmatrix},
\] 
consisting of two Jordan blocks of size $3$ and $5$. A computation using the matrix representation of the quadratic form $\mathfrak{q}$ shows that $\tilde{g}$ does not fix any $1$-dimensional non-degenerate subspace of $V$ and hence the projective image of $\tilde{g}$ in $\mathrm{P}\Omega_8^+(q)$ does not lie in any subgroup of the form $\nor M{\mathrm{P}\Omega_7(q)}$ in the Aschbacher class $\mathcal{C}_1$. Similarly, using the information on the spin representations of $\Omega_7(q)$ in Section~5.4 of~\cite{kl}, we deduce that the projective image of $\tilde{g}$ in $\mathrm{P}\Omega_8^+(q)$ does not lie in any subgroup of the form $\nor M{\mathrm{P}\Omega_7(q)}$ in the Aschbacher class $\mathcal{S}$. In particular, if we let $g\in\mathrm{P}\Omega_8^+(q)$ be the projective image of $\tilde{g}$, then $g$ does not lie in any conjugate of $M\cap H_\omega$ or of $M\cap H_{\omega'}$, contradicting~\eqref{eq:new}.

Suppose $(M,M\cap H_\omega)$ is as in line~6. Observe that we must also have $\mathrm{PSp}_2(q^2)\unlhd M\cap H_{\omega'}$. Now,~\cite[Tables~$8.12$ and~$8.14$]{bhr}  lists all the maximal subgroups of $\mathrm{PSp}_4(q)$. (There are two tables to consider depending whether $q$ is odd or $q$ is even, because when $q$ is even $\mathrm{PSp}_4(q)$ admits a graph-field automorphism.) From the ``$c$ column'' in~\cite[Tables~$8.12$ and~$8.14$]{bhr}, we deduce that  $M\cap H_\omega$ and $M\cap H_{\omega'}$ are conjugate in $M$. Now,~\eqref{eq:new} contradicts Jordan's theorem.

This concludes the proof of Theorem~\ref{thrm:main1}.

\section{Triangles in derangement graphs}\label{sec:triangles}

In this section we prove Theorem~\ref{thrm:main2}. In what follows, we assume that Theorem~\ref{thrm:main2} is false and we let $G$ a counterexample to Theorem~\ref{thrm:main2} with $|G|+|\Omega|$ as small as possible. We consider two cases.

\subsection{Case 1: $G$ acts primitively on $\Omega$}

Throughout this subsection $N$ denotes the socle of $G$. Broadly speaking, the O'Nan-Scott theorem classifies the finite primitive groups and specifically, it describes in detail the embedding of $N$ in $G$ and collects some useful information about the action of $N$. The main theorem of~\cite{LPSLPS} is the O'Nan-Scott theorem for finite primitive permutations groups. In this work five types of primitive groups are defined (depending on the group- and action-structure of the
socle), namely the \textbf{\textit{Affine-type}} (HA), the \textbf{\textit{Almost Simple}} (AS), the \textbf{\textit{Diagonal-type}}, the  \textbf{\textit{Product-type}}, and the \textbf{\textit{Twisted Wreath product}}, and it is shown that  every primitive group belongs to exactly one of these types.  
In~\cite{18} this division into types is refined further, namely the Diagonal-type is partitioned in \textbf{\textit{Holomorphic simple}} (HS), and \textbf{\textit{Simple Diagonal}} (SD), and the Product-type is partitioned into \textbf{\textit{Holomorphic compound}} (HC), \textbf{\textit{Compound Diagonal}} (CD), and \textbf{\textit{Product action}} (PA).

First we need an important definition. A finite transitive permutation group $G$ acting on a set $\Omega$ is \textbf{\textit{$2'$-elusive}} if 
\begin{itemize}
\item $|\Omega|$ is divisible by an odd prime, 
\item $G$ does not contain a derangement of odd prime order.
\end{itemize}

 If $G$ contains a derangement $g$ of odd prime order, then $\langle g\rangle$ is a clique in $\Gamma_G$ of size greater than $2$, which contradicts the fact that $\Gamma_G$ has no triangles. Therefore, either $G$ is $2'$-elusive or $|\Omega|$ is a power of $2$. 

Assume that $G$ is $2'$-elusive. The main result in~\cite{BG} shows that all of the following hold:
\begin{itemize}
\item $G$ is either AS or PA ; 
\item $\Omega=\Delta^k$ admits a product structure with $k\ge 1$;
\item $N=\mathrm{soc}(L)^k \unlhd G\le L\, \mathrm{wr} \, K$, where $L\le \Sym(\Delta)$ is primitive of AS type, $K\le \Sym(k)$ is transitive, and $G$ is endowed of the product action; and 
\item either $L=M_{11}$ and $|\Delta|=12$, or $\mathrm{soc}(L)={}^2F_4(2)'$ and $|\Delta|=2\, 304$.
\end{itemize}
It can be easily checked with \texttt{magma} that, when $L=M_{11}$ and $|\Delta|=12$, $L$ contains an element $x$ of order $8$ with the property that $x^2$ is a derangement. Therefore the set $\{1, x, x^2\}$ is a clique of size $3$ in $\Gamma_L$. Now, 
$$(\underbrace{1,\ldots,1}_{k\textrm{-times}})\,,(\underbrace{x,\ldots,x}_{k\textrm{-times}})\,(\underbrace{x^2,\ldots,x^2}_{k\textrm{-times}})$$
belong to $L^k\le G$ and form a clique of size $3$ in $\Gamma_G$. However, this contradicts the fact that $G$ is a counterexample to Theorem~\ref{thrm:main2}. Similarly,  it can be easily checked with \texttt{magma} that, when $\mathrm{soc}(L)={}^2F_{4}(2)'$ and $|\Delta|=2\,304$, $\mathrm{soc}(L)$ contains an element $g$ of order $4$ with the property that $g^2$ is a derangement. Therefore, the set $\langle g\rangle$ is a clique of size $4$ in $\Gamma_L$. As above, this clique can be used to obtain a clique of size $4$ in $\Gamma_G$, contradicting the fact that $G$ is a counterexample to Theorem~\ref{thrm:main2}.

Next we consider the case  where $|\Omega|$ is a power of $2$. From the O'Nan-Scott theorem, $G$ is either HA, AS
or PA. If $G$ is HA, then $N=\mathrm{soc}(G)$ is a regular subgroup of $G$. As $|N|=|\Omega|\ge 3$, $\Gamma_G$ has a clique of size at least 3, contradicting the fact that $G$ is a counterexample to Theorem~\ref{thrm:main2}. In particular, $G$ is either AS or PA. Using the structure of primitive groups of AS and PA type and using the main result in~\cite{Bob}, we deduce that all of the following hold:
\begin{itemize}
\item $\Omega=\Delta^k$ admits a product structure with $k\ge 1$;
\item $N=\mathrm{soc}(L)^k\unlhd G\le L\,\mathrm{wr}\,K$, where $L\le \Sym(\Delta)$ is primitive of AS type, $K\le \Sym(k)$ is transitive, and $G$ is endowed of the product action; and
\item either $\mathrm{soc}(L)=\Alt(\Delta)$, or 
                  $\mathrm{soc}(L)=\mathrm{PSL}_n(q)$, $n$ is prime, $|\Delta|=\frac{q^n-1}{q-1}$ and the action of $L$ on $\Delta$ is the natural action on the points of the projective space.
\end{itemize}

When $\mathrm{soc}(L)=\Alt(\Delta)$, we must have $|\Delta|\ge 8$, because $\Alt(4)$ is not a non-abelian simple group. Let $x\in \mathrm{soc}(L)=\Alt(\Delta)$ be the product $|\Delta|/4$ disjoint cycles of length $4$. Then $\langle x\rangle$ is a clique of size $4$ in $\Gamma_L$. As above, this clique can be used to obtain a clique of size $4$ in $\Gamma_G$, contradicting the fact that $G$ is a minimal counterexample to Theorem~\ref{thrm:main2}. 

Assume then $\mathrm{soc}(L)=\PSL_n(q)$. Now, Zsigmondy's theorem~\cite{zi} shows that $|\Delta|=(q^n-1)/(q-1)$ is a power of $2$ only when $n=2$ and $q$ is a Mersenne prime, that is, $q=2^\ell-1$, for some  $\ell\in\mathbb{N}$. Observe that $\ell\ge 3$ because $\PSL_2(3)$ is not a non-abelian simple group. Thus 
\[
|\Delta|=\frac{q^n-1}{q-1}=q+1 = 2^\ell \ge 8.
\]
Now, a Singer cycle $C$ in $\PSL_2(q)$ has order $(q+1)/2=2^{\ell-1}\ge 4$ and has two orbits on $\Delta$ of cardinality $(q+1)/2$. Therefore, $C$ is a clique of size $(q+1)/2\ge 3$ in $\Gamma_L$. This clique can then be used to obtain a clique of size $(q+1)/2$ in $\Gamma_G$, contradicting again the fact that $G$ is a counterexample to Theorem~\ref{thrm:main2}.

This concludes the analysis when $G$ is primitive.

\subsection{Case 2: $G$ acts imprimitively on $\Omega$}

Fix $\omega \in \Omega$ and let $H$ be a subgroup of $G$ with $G_\omega < H < G$. Observe that this is possible because $G$ is not primitive on $\Omega$ and hence $G_\omega$ is not a maximal subgroup of $G$. Now, let $\Delta$ be the system of imprimitivity determined by the overgroup $H$ of $G_\omega$. As $|\Delta|<|\Omega|$, from our inductive argument we deduce that, if $|\Delta|\geq 3$, then the derangement graph of the permutation group induced by the action of $G$ on $\Delta$ has a triangle. In this case, it follows that the derangement graph $\Gamma_G$ has a triangle, which contradicts the fact that $G$ is a counterexample to Theorem~\ref{thrm:main2}. Therefore $[G:H]=|\Delta|=2$. As $H\unlhd G$ and $G_\omega \le H$, we deduce that $G_\alpha \le H$, for every $\alpha \in \Omega$. In particular,
$$\bigcup_{\alpha \in \Omega}G_\alpha\subseteq H$$
and hence every element of $G\setminus H$ is a derangement.

Suppose $H$ contains a derangement $h$. Let $g\in G\setminus H$. Then $g$ and $hg$ are derangements, because $g,gh\in G\setminus H$. Now,
\[
\{1,\, g,\,hg\}
\]
is a triangle in the derangement graph $\Gamma_G$, which contradicts the fact that $G$ is a counterexample to Theorem~\ref{thrm:main2}. Therefore, $H$ contains no derangements, that is, 
\[
H=\bigcup_{\alpha \in \Omega}G_\alpha.
\]
From this it follows that the derangements of $G$ are exactly the elements of $G\setminus H$ and that $\Gamma_G$ is a complete bipartite graph with bipartition $\{H,G\setminus H\}$. Theorem~\ref{thrm:main1} yields $|\Omega|=2$, contradicting the fact that $|\Omega|\ge 3$. This concludes the proof of Theorem~\ref{thrm:main2}.

We conclude this section giving a corollary to Theorem~\ref{thrm:main2}.
\begin{corollary}\label{cor:1}If $G$ is transitive of degree $n\ge 3$, then $\rho(G)\leq \frac{n}{3}$.
\end{corollary}
\begin{proof}
If the derangement graph for a group $G$ has a clique of size $k$, then from the clique-coclique bound we deduce $\alpha(\Gamma_G) \leq \frac{|G|}{k}$ and $\rho(G) \leq \frac{n}{k}$. Now, the result follows from Theorem~\ref{thrm:main2}.
\end{proof}

\section{Examples Tripartite and Multipartite Derangement graphs}\label{sect:counter-example}

In this section we give examples of transitive groups having derangement graph that is complete tripartite or multipartite. Among other things, these show that the bound given in Corollary~\ref{cor:bound} is tight.

All transitive groups of degree at most $48$ have been determined by the work of Cannon, Holt, Hulpke and Royle, see~\cite{CH,HR,MR2168238}. These groups are available, for instance, in the computer algebra system \texttt{magma}~\cite{magma}. 
\begin{theorem}\label{thrm:enu} Up to degree $48$ there are four transitive groups $G$ having derangement graph complete tripartite. Using the numbers $(n,d)$ in the database of \texttt{TransitiveGroups} in the version V2.25-5 of \texttt{magma}, these are
\begin{enumerate}\label{thm:tripartites}
\item  $(6,4)$ having degree $6$ and order $12$,
\item $(18,142)$ having degree $18$ and order $324$, 
\item $(30,126)$ having degree $30$ and order $600$, 
\item $(30,233)$ having degree $30$ and order $1\,200$. 
\end{enumerate}
\end{theorem}
\begin{proof}
This follows from an exhaustive computer search.
\end{proof}

It was conjectured by Li, Song and Pantangi~\cite[Conjecture 1.2]{li2020ekr} that, if $G$ is transitive of degree $n$, then $\rho(G)< \sqrt{n}$. However, this turns out to be false because the second group in Theorem~\ref{thm:tripartites} has $\rho(G) =  6>\sqrt{18}$.
Similarly, the third and the forth group in Theorem~\ref{thm:tripartites} have $\rho(G) = 10>\sqrt{30}$.

We have defined the EKR property and the strict-EKR property, but there is a third important property for intersecting permutations in a permutation group. A transitive group has the \textbf{\textit{EKR-module property}} if the characteristic vector of any maximum intersecting set of permutations is a linear combination of the characteristic vectors of the canonical intersecting sets. It is proven in~\cite{MeagherSin} that every 2-transitive group has the EKR-module property.  The next result shows that there are many  groups having the EKR-module property, but not the EKR-property.

\begin{theorem}
Let $G$ be transitive and  suppose that $H_G$ (as defined in~\eqref{eq:HG}) is a proper derangement-free subgroup of $G$. Then $G$ has the EKR-module property. 
\end{theorem}
\begin{proof}
Let $\Omega$ be the domain of $G$. 
Since $H_G$ is derangement free, by Jordan's theorem, $H_G$ is intransitive. Let $\omega\in \Omega$ and let $O$ be the orbit of $H_G$ containing $\omega$. 

For any $\omega^\prime  \in O$, there is an element $h \in H_G$ that maps $\omega$ to $\omega^\prime$. As $G_\omega \le H_G$, 
\[
hG_\omega = G_{\omega \rightarrow \omega^\prime}\subseteq H_G,
\] 
where $G_{\omega \rightarrow \omega^\prime}$ is the set of all permutations in $G$ that map $\omega$ to $\omega^\prime$.
The sets $G_{\omega\rightarrow \omega^\prime}$ for $\omega^\prime \in O$ are pairwise disjoint and hence 
\[
\left|\bigcup_{\omega' \in O} G_{\omega \rightarrow \omega'}  \right| =|O||G_\omega|=[H_G:G_\omega]|G_\omega| = |H_G|.
\]
Thus $H_G$ is the union of the canonical cocliques that map $\omega$ to $\omega'$, where $\omega'$ runs through the elements in $O$.
\end{proof}
It follows from this result that the groups in Theorem~\ref{thrm:enu} all have the EKR-module property, because in each of these permutation groups $H_G$ is a proper derangement-free subgroup of $G$.

At present, the only  transitive groups having derangement graph complete tripartite are the four given in Theorem~\ref{thrm:enu}, but there are many examples of transitive groups with a complete multipartite derangement graph.

\begin{lemma}
Let $n$ be an even integer with $n/2$ odd and $n\ge 6$. Then there is a transitive group $G$ of degree $n$ with $\Gamma_G$ complete multipartite with $n/2$ parts.
\end{lemma}

\begin{proof}
Let 
\begin{align*}
H&:=\langle(1,2), (3,4), \ldots, (n-1,n)\rangle,\\
G &:= \Alt(n) \cap \langle H, (1,3,\dots, n-1)(2,4,\dots, n) \rangle .
\end{align*}

Since $n/2$ is odd and each permutation in $H \cap \Alt(n)$ is the product of an even number of transpositions, the subgroup $H \cap \Alt(n)$ is a coclique in $\Gamma_G$. Moreover, every element in $G \backslash (H \cap \Alt(n))$ is a derangement. Since $\Gamma_G$ is vertex transitive and $[G:H\cap\Alt(n)]=n/2$, $\Gamma_G$ is  complete multipartite with $n/2$ parts.
\end{proof}

\section{Future Work}\label{sec:future}

From Corollary~\ref{cor:1}, we have $I(n) \leq \frac{n}{3}$ and, from Theorem~\ref{thrm:enu}, $I(n) =\frac{n}{3}$, when $n \in \{6,18,30\}$.  We have not been able to find a general construction for transitive groups $G$ of degree $n$ with $\rho(G) = \frac{n}{3}$ and it is not clear to us if infinitely more examples exist. From Theorem \ref{thrm:main1}, if $G$  is transitive of degree $n$ and $\Gamma_G$ is bipartite, then $n \leq 2$. These two facts naturally lead to the following.

\begin{question}
Let $G$ be transitive of degree $n$ with $\Gamma_G$ $k$-partite. 
Is there an upper bound on $n$ as a function of  $k$ only?
\end{question}

We are inclined to believe that  tripartite derangement graphs of transitive groups are very special, in the sense that is given in the following conjecture.

\begin{conjecture}
If $G$ is transitive of degree $n$ with intersection density $\frac{n}{3}$ and with $\Gamma_G$ connected, then $\Gamma_G$ is complete tripartite.
\end{conjecture}

For each $n\ge 3$, the set $\mathcal{I}_n$ is a finite list of rational numbers between 1 and $\frac{n}{3}$. We suspect that it is rare for the upper bound of $\frac{n}{3}$ to be reached. 

\begin{problem} 
For any $n$, determine $I(n)$ as an explicit function of $n$.
\end{problem}

Further, we also ask the following.
\begin{question}
For each $n$, what is the structure of the transitive groups $G$ of degree $n$ with $\rho(G) = I(n)$?
\end{question}

Here we observe that, in searching for the maximum value in $\mathcal{I}_n$ (that is, $I(n)$), it suffices to only consider the minimally transitive subgroups. 

\begin{lemma}
Let $G$ and $H$ be transitive groups with $H \leq G$. Then $\rho(G) \leq \rho(H)$. 
\end{lemma}
\begin{proof}
Since $H$ is a subgroup of $G$, the graph $\Gamma_H$ is an induced subgraph of $\Gamma_G$; this means that there is a graph homomorphism of $\Gamma_H$ to $\Gamma_G$. The  ``No-Homomorphism Lemma''~\cite[Theorem~2]{NoHom} implies that
\[
\frac{| H|}{\alpha(\Gamma_H)} \leq \frac{|G|}{\alpha(\Gamma_G)}.
\]
Rearranging the terms in this inequality, we obtain the result.
\end{proof}

Based on the computational evidence on the transitive groups of degree at most $48$, we have compiled a list of conjectures for $I(n)$.

\begin{conjecture}\hfill
\begin{enumerate}
\item If $n$ is even, but not a power of $2$, then there is a transitive group $G$ of degree $n$ with $\Gamma_G$ a complete multipartite graph with $n/2$ parts.
\item If $n$ is a prime power, then $I(n) = 1$.
\item If $n =pq$ where $p$ and $q$ are odd primes, then $I(n) = 1$.
\item If $n = 2q$ where $q$ is prime, then $I(n) = 2$. 
\end{enumerate}
\end{conjecture}

We also have the more general problem.

\begin{problem}
Find more examples of transitive groups having complete multipartite derangement graph. Is there an insightful characterization of these groups?
\end{problem}


It is easy to see that, if $G$ is transitive of degree $n$, then $\Gamma_G$ is $n$-partite. Indeed, if $\Omega$ is the domain of $G$ and $\omega\in \Omega$, then the vertex set of $\Gamma_G$ is partitioned into the $n$ canonical intersecting families $\{G_{\omega\to \omega'}\mid \omega'\in \Omega\}$, which are cocliques of $\Gamma_G$. Thus $\Gamma_G$ is $n$-partite.
This leads to several general questions.

\begin{question} \hfill
\begin{enumerate}
\item Find transitive groups $G$ of degree $n$ with $\Gamma_G$ a $k$-partite graph, with $k<n$.
\item Can we describe the structure of $k$-partite derangement graphs?
\item For which values of $(n,k)$ does there exist a transitive group $G$ of degree $n$ with $\Gamma_G$ a (complete) $k$-partite graph? \label{part3}
\end{enumerate}\label{quest:cmp}
\end{question}

Regarding Question~\ref{quest:cmp}~\eqref{part3}, complete $k$-partite derangement graphs for transitive groups do not always exist, even when $k \mid n$. For instance, there is  no transitive group of degree $9$ having derangement graph tripartite.

 Li, Song and Pantnagi~\cite{li2020ekr} proved that for any $M$ and $\varepsilon \in  (0,1)$, there exists a transitive group $G$ acting on a set $\Omega$ with $\rho(G) >M$ and $(1-\varepsilon)\sqrt{|G|} < \rho(G)$.  Their proof gives an example of such a group, and the group they give is a quasi-primitive group. So in considering the intersection density it might be useful to consider imprimitive, quasi-primitive and primitive groups separately. We end with a final question in this direction.
 
\begin{question}
What is the maximum value $\rho(G)$ when $G$ is a quasi-primitive group?
\end{question}


\end{document}